\documentclass[10pt]{article}
\font\smallit=cmti10
\font\smalltt=cmtt10

\usepackage{amssymb,latexsym,amsmath,epsfig,amsthm} 
\usepackage{subcaption}
\usepackage{multicol}
\usepackage{listings}  
\usepackage{algorithm}
\usepackage[noend]{algpseudocode}
\usepackage{color}
\usepackage{array}   
\newcolumntype{L}{>{$}l<{$}}

\newcommand{\mul}{\times}
\newcommand{\ccat}{\oplus}
\newcommand{\np}{}

\newcommand{\para}[1]{\vspace*{.2cm}\noindent\textbf{#1}. }

\makeatletter

\renewcommand\section{\@startsection {section}{1}{\z@}
{-30pt \@plus -1ex \@minus -.2ex}
{2.3ex \@plus.2ex}
{\normalfont\normalsize\bfseries\boldmath}}

\renewcommand\subsection{\@startsection{subsection}{2}{\z@}
{-3.25ex\@plus -1ex \@minus -.2ex}
{1.5ex \@plus .2ex}
{\normalfont\normalsize\bfseries\boldmath}}

\renewcommand{\@seccntformat}[1]{\csname the#1\endcsname. }

\makeatother

\newtheorem{lemma}{Lemma}

\newtheorem{definition}{Definition}

\definecolor{mygreen}{rgb}{0,0.6,0}
\definecolor{mygray}{rgb}{0.5,0.5,0.5}
\definecolor{mymauve}{rgb}{0.58,0,0.82}

\begin{document}

\begin{center}
\uppercase{\bf Crazy Sequential Representations of Numbers for Small Bases}
\vskip 20pt
{\bf Tim Wylie\footnote{This research was supported in part by National Science Foundation Grant CCF-1817602.}}\\
{\smallit University of Texas - Rio Grande Valley\\ Edinburg, TX 78539-2999, USA}\\
{\tt timothy.wylie@utrgv.edu}\\ 
\end{center}
\vskip 30pt


\centerline{\bf Abstract}

\noindent
Throughout history, recreational mathematics has always played a prominent role in advancing research. Following in this tradition, in this paper we extend some recent work with crazy sequential representations of numbers$-$ equations made of sequences of one through nine (or nine through one) that evaluate to a number. All previous work on this type of puzzle has focused only on base ten numbers and whether a solution existed. We generalize this concept and examine how this extends to arbitrary bases, the ranges of possible numbers, the combinatorial challenge of finding the numbers, efficient algorithms, and some interesting patterns across any base. For the analysis, we focus on bases three through ten. Further, we outline several interesting mathematical and algorithmic complexity problems related to this area that have yet to be considered.

\pagestyle{myheadings}
\markright{\smalltt \hfill}
\thispagestyle{empty}
\baselineskip=12.875pt
\vskip 30pt

\section{Introduction}

One constant theme throughout the history of mathematics is the lure of and the desire to create and solve puzzles. Countless areas of research have been created and extended based on an investigation into recreational mathematics. The study of games and puzzles has become a serious area in its own right often providing insights into much deeper topics.

In this paper we look at an area of recreational mathematics based in number theory and combinatorics began in 2013 by Taneja \cite{Taneja:2013:1302.1479v1} and continued in  \cite{Taneja:2013:1302.1479v3,Taneja:2013:1302.1479v2,Taneja:2013:1302.1479v4,Taneja:2013:1302.1479v5}. The crazy sequential representation of a number is an arithmetic expression, equal to the value of the number, that contains the digits of a base in order (ascending or descending) such as 
\[
\hspace*{.9cm} 3227 = 123 + 45 \mul 67+ 89 \hspace*{.3cm} \mbox{ and } \hspace*{.3cm} 3227 = 9 + 87 + 6 + 5^4 \mul (3+ 2)\mul 1.
\]
This representation is often not unique. The original work looked at expressions with only addition and multiplication as well as concatenation and exponentiation\footnote{Taneja used the term `potentiation' instead, which comes from the translated word used for exponents.}. Taneja extended this work by also allowing subtraction and division, and was able to find equations for all numbers $1-11111$ with one exception: an ascending equation for 10958. Without concatenation and exponentiation, we could look at group operations to define possible values, but these two operations do not provide closure.

There are examples of this kind of representation of integers at least as far back as 1917 in a famous puzzle book by Dudeney \cite{Dudeney:1917:BOOK}, and also in another recreational book by Madachy \cite{Madachy:1966:BOOK} from 1966. Both of these works only focused on the number 100 and used other operations such as factorials and square roots, as well as decimals, etc. Taneja was unaware of these books in his original work, and discovered them later while working on the updated version.

Our focus in this work is to look at possible numbers in other bases- specifically bases less than 10. We also summarize the work related to base 10 and give an exhaustive proof that under Taneja's rules, 10958 is indeed impossible. We follow previous convention and only allow addition, concatenation, exponentiation, multiplication, division, and negation\footnote{Taneja specified subtraction, but we use a broader operation, and we show that arbitrary negation is still not sufficient for 10958.} along with precedence constraints (parentheses). 

We can examine the limitations of specific operations, and how the possible results are affected by a change of base. Here, we focus on what is possible within a given base. As an example, Table \ref{tab:overview} shows for each base less than 10 the first positive integer that is impossible for increasing and decreasing representations as well as the first positive integer that can not be sequentially represented either increasingly or decreasingly. 


\begin{table}[t]
    \centering
    \footnotesize
    \renewcommand{\arraystretch}{1.2}
    \begin{tabular}{| c | c | c | c |}
        \hline
        \textbf{Base} & \textbf{Increasing} & \textbf{Decreasing} & \textbf{Neither} \\ \hline
        3 & 0 & 0  & 0\\ \hline
        4 & 13& 11  & 16\\ \hline
        5 & 27& 17 & 27\\ \hline
        6 & 67& 77  & 92\\ \hline
        7 & 260 & 262 & 292\\ \hline
        8 & 614 & 809  & 1192 \\ \hline
        9 & 3293 & 4570  & 5414 \\ \hline
        10 & 10958 & 14324  & 21212 \\ \hline        
    \end{tabular}
    \caption{A brief overview of the first integers that can not be sequentially represented under the defined operations for bases  $3-10$.} \label{tab:overview}
\end{table}


Historically, these kind of derivations were done tediously and slowly, and Taneja's work also has this flavor with only using a program to find a few of the difficult numbers \cite{Taneja:2013:1302.1479v5}. Our approach leverages modern computational power and algorithmic techniques to bring this topic squarely within computational mathematics and search all possible combinations. We discuss these techniques and upper bounds in the paper. 

A brute force approach to a problem like this has generally been classified as computational mathematics - there is a point for many problems at which the number of possible combinations becomes too large for a human, or humanity, to check by hand in any reasonable amount of time. This has become more common with efforts to verify and prove other long open questions in mathematics such as the Kepler Conjecture \cite{Hales:2006:DCG,flyspeck:2014:WEB,Hales:2008:CORR}, the Boolean Pythagorean Triples problem \cite{Heule:2016:LNCS}, finding Ramsey numbers \cite{Graham:1990:BOOK,Radziszowski:2014:EJOC,Radziszowski:1995:JGT}, the Happy Ending problem \cite{Morris:2000:BAMS,Szekeres:2006:ANZIAM}, the 2-PATS problem \cite{Kari:2017:A}, and many others where brute-force exhaustive-search solutions were required.

Fortunately for us, this problem can also be approached with dynamic programming through calculating substrings that appear in multiple equations. This recurrence relation yields an efficient solution allowing an exhaustive examination within a reasonable amount of time. For most of the bases in our study (3-9), even basic laptops are sufficient to check the millions of combinations. For base ten, we utilized some research servers due to the high memory requirements. The program required around 20 gigabytes of memory to run, but the the time was less than two hours.

In the next section we give the background and definitions necessary. We then overview the approach and algorithms used in this research in Section \ref{sec:algorithms}. We discuss the small bases 2, 3, and 4 in Section \ref{sec:toosmall}, and then the more substantial possibilities of bases 5-9 in Section \ref{sec:small}. Section \ref{sec:b10} covers what is known about base 10 and the missing number 10958.
Finally, in Section \ref{sec:fun} we outline several interesting mathematical and computational open problems related to their study and conclude.

\section{Preliminaries} \label{sec:preliminaries}
We generalize the previous definitions with negation instead of subtraction, an explicit concatenation operator, and adding parentheses. 

\begin{definition}[Crazy Sequential Representation]
Given a number $n \in \mathbb{R}$, an increasing crazy sequential representation of $n$ in base $b$ is an equation using the sequence of numbers $\langle 1,2,\dots ,b-1 \rangle$ (decreasing being $\langle b-1, \dots , 2, 1 \rangle$) with the following operations allowed between any two of the numbers.
Given two real numbers $x,y \in \mathbb{R}$ we define the following allowable operations:
    \begin{itemize}
        \item[$+$] Addition: $x + y$ resulting in the sum of the two numbers.
        \item[$-$] Negation: $-x$ is allowable as well as the negation of an expression $-(\dots)$.  Addition with a negative is also equivalent to subtraction in this context, so subtraction is omitted from the list of operations.
        \item[$\mul$] Multiplication: $x \mul y$ resulting in their product.
        \item[$/$] Division: $x/y$ giving the fraction.
        \item[$a^b$] Exponentiation: $x^y$ meaning $x$ to the $y^{th}$ power.
        \item[$xy$] Concatenation: $xy$ meaning the number $xy$ in the given base (e.g., $12_3=5_{10}$ ). There are many standard symbols used for this operation. We will use $\ccat$ when we need to explicitly show it, otherwise it will be omitted when clear by context- generally $xy$ will be preferred instead of $x\ccat y$. 
        \item[$()$] Grouping: arbitrary parentheses are allowed with derivations following the standard rule that expressions inside parentheses are evaluated first.
    \end{itemize}
\end{definition}
%


One goal of Taneja's work is to minimize the number of operations used for a given representation. Thus, the original work \cite{Taneja:2013:1302.1479v1} focused on numbers derivable from simply concatenation, addition, multiplication and exponentiation. Later work to add missing numbers included division and subtraction  \cite{Taneja:2013:1302.1479v3,Taneja:2013:1302.1479v2,Taneja:2013:1302.1479v4,Taneja:2013:1302.1479v5}. We have also opted to generally prefer those original operations in the expression chosen when multiple expressions exists for a given number, as well as simplicity and elegance.

\para{Explicit Concatenation}
An issue with the way Taneja uses concatenation is that it is only allowed before evaluating the expression. This means $12$ is allowed as twelve (or $1\ccat 2$), but $(1 + 2) \ccat 3$ is not allowed to be evaluated as 33. This is the only defined operation not allowed during evaluation. If we allow it, several other numbers are possible, including 10958 in base 10. In the results, all expressions using this deviant version are colored red and use the $\ccat$ symbol explicitly. Our approach did not consider these solutions either, and thus there may be solutions of this form to some of the values listed without a solution. 

\subsection{Combinatorics}


In calculating an upper bound we are looking at the maximum amount of different numbers that could be represented in that base.
The number of parse trees that can be generated with binary operators tells us the number of ways to distribute the operations. If we, for the moment, only consider a single operation, this is the well-known Catalan numbers. Another view more relevant is the number of ways to insert $n-1$ pairs of parentheses in a word of $n$ letters. e.g., for $n=3$ ($t(2)$) there are 2 ways: $((ab)c)$ or $(a(bc))$ \cite{Guy:1973:AMM}. The Catalan numbers can be recursively derived by the following equation with $t(0)=1$ and $t(1)=1$.

\begin{equation}
t(n) = \sum_{i=1}^n t(i-1)t(n-i)
\end{equation}

Thus, for the bases considered here, we have $t:(2, 3,\dots ,9) \implies
(2, 5$, $14, 42$, $132$, $429$, $1430,$ $4862)$. 
This gives the number of ways to group the operands (sequential numbers), and then we must consider the number of operators allowed. We allow five distinct operations as defined above: addition, multiplication, division, exponentiation, and concatenation (subtraction will be handled later). This gives $5^{n-1}$ ways to place the operations on $n$ operands. For base $b$, we therefore have $5^{b-2}$ since we exclude 0 in the representation and only use $1, \dots, (b-1)$.

The last issue to deal with is negation. If we only allow subtraction, then the number of operations is $6^{n-1}$, however, we also allow negation. Thus, expressions such as $-(-4+5)$ are also allowed. Thus, for each of the parentheses or numbers, we could negate it, which adds all possible combinations of negations over the parentheses and numbers. This means we can also reduce our operations to only 5 (since we will look at adding the negated number instead). Thus, we have the power set of $n$ possible ways to add negatives to the numbers for $n$ operands, and the power set of $\{1,2,\dots,n-1\}$ for possible ways to add negatives to the parentheses (for $n$ numbers, we need at most $n-1$ parenthesis for binary operations). Since for base $b$, we have $n=b-1$, when we include all possibilities, there is an upper bound for the   
combinations for $n$ numbers given base $b$.

\begin{align}
C(n) &= 5^{n-1} \mul t(n) \mul 2^{n} \mul 2^{n-1} \\
&=5^{n-1} \mul t(n) \mul 2^{2n-1}, \mbox{           or in terms of } b\\
C(b) &= 5^{b-2} \mul t(b-1) \mul 2^{2b-3}.
\end{align}

The values for bases $3-10$ are shown in Table \ref{tab:basecomb}. Note that the vast majority of these combinations do not yield integers, however, the numbers are small enough to output all possible numbers and then check the integral ones. 
Many of these results are duplicates with only parenthetical differences, but the number of combinations is still well within computational power to brute force every possibility even if many are duplicates. For larger bases, an examination of the unique parse trees would reduce many of the duplicates caused by analyzing strings. 

\begin{table}[t]
    \centering
    \footnotesize
    \renewcommand{\arraystretch}{1.4}
    \begin{tabular}{|r| c | c |c |c |} \hline
        \textbf{Base $b$}           & \textbf{3}  & \textbf{4}    & \textbf{5}                 & \textbf{6}     \\ \hline
        \textbf{Combinations $C(b)$}   & 80 & 4000 & $2.24 \times 10^5$  & $1.344 \times 10^7$  \\ \hline 
        \multicolumn{5}{r}{ } \\ \hline
        \textbf{Base $b$} & \textbf{7}       & \textbf{8}        & \textbf{9}         & \textbf{10} \\ \hline
        \textbf{Combinations} $C(b)$ & $8.448 \times 10^{8}$ & $5.4912 \times 10^{10}$ &  $3.6608 \times 10^{12}$  & $2.489344 \times 10^{14}$ \\  
        \hline
    \end{tabular}
    \caption{The upper bound on the number of combinations for crazy sequential representations for a given base, which is the maximum amount of possible numbers that could be represented.}
    \label{tab:basecomb}
\end{table}

\section{Algorithms} \label{sec:algorithms}
At a high-level, in order to find all the numbers possible for a given base, an algorithm such as Algorithm \ref{alg:finddp} can be run. List the numbers from 1 to $b-1$ (or $b-1$ down to 1), and then check for all valid expressions with the given operations. This includes both the removal of any operation (concatenation) and the possibility of precedence in operations (parentheses). 


There are several notes of interest related to actual implementation. These include finding all binary partitions (and how this changes with concatenation), negation of terms, evaluation in the given base, and processing such large amounts of data. We cover these in the analysis of Algorithm \ref{alg:finddp}, which is a dynamic programming solution to the problem. By utilizing a dictionary of substrings, we can exponentially reduce the number of computations necessary.

\begin{algorithm}[t]
\footnotesize
\begin{algorithmic}[1]
 \Function {FindExpressions}{$base, low, high \in \mathbb{Z}^+$}
    \If{low $\neq$ high}
        \State $T = \{\}$ 
        \State numstr $\leftarrow$ CASTSTR($low$) $\ccat \dots \ccat$ CASTSTR($high$)
        \State catnum $\leftarrow$ CASTNUM(numstr)
    \State $T \leftarrow T \cup (catnum, numstr) \cup (-catnum, $``$ - $''$\ccat numstr)$
   
    \ForAll{$low \leq k \leq high$ }
        \State $L \leftarrow$ FindEpressions(base, low, k)
        \State $R \leftarrow$ FindEpressions(base, k+1, high)
        
        \State $\triangleright$ All ways to combine the left and right expressions
        \ForAll{$x \in LS$}
            \ForAll{$y \in RS$}
                \State $T \leftarrow T \cup (x+y, $`$($'$\ccat$ $L_x$ $\ccat $`$+$'$ \ccat$ $R_y$ $\ccat$`$)$'$)$
                \State $T \leftarrow T \cup (x\mul y, $`$($'$\ccat$ $L_x$ $\ccat $`$\mul$'$ \ccat$ $R_y$ $\ccat$`$)$'$)$
                \State $T \leftarrow T \cup (x/y, $`$($'$\ccat$ $L_x$ $\ccat $`$/$' $\ccat$ $R_y$ $\ccat$`$)$'$)$\Comment{if $y \neq 0$}
                \State $T \leftarrow T \cup (x^y, $`$($'$\ccat L_x \ccat $`\^\ '$ \ccat R_y \ccat$`$)$'$)$
            \EndFor
    
        \EndFor
        
    \EndFor
    \Return $T$
    \EndIf
 \EndFunction
 \\
 
 $F = $FindEpressions(10, 1, 9)
\end{algorithmic}
\caption{A recursive algorithm looking at the possible combinations using dynamic programming that builds a dictionary or lookup table of all expressible numbers.}
\label{alg:finddp}
\end{algorithm}

\para{Finding Possible Parentheses}
The possible ways parentheses can be nested for $n$ items is a classic problem in Computer Science with the proof published by Guy and Selfridge in 1973 \cite{Guy:1973:AMM}. 
An example of a Python algorithm to generate these is here \cite[btilly]{web:para}.



\para{Finding Negations}
Given all possible nested parentheses, for each we need to find all possible negations of the numbers and the individual expressions. With negation instead of subtraction, the following are all different: $(((-1+\dots$,  $-(((1+\dots$, $(-((1+\dots$, and $((-(1\dots$. 

\para{Coding with Bases}
Another small implementation detail is the need to deal with switching between multiple bases, which python has a method within casting to do so 234 in base 7 would be float(int(237,7)). 

\section{Too Small Bases} \label{sec:toosmall}

This is a quick overview of bases that are really too small to offer the necessary flexibility to count very high. Namely, 2, 3, and 4. Five could be in this category, but there is a massive jump between 4 and 5, so we will put it with the larger bases.

\para{Base 2}
For base 2, since we do not use $0$, only operations on the single digit $1$ can be performed, meaning $1$ and $-1$ are the only numbers expressible in a sequential representation. Thus, we can ignore it.

\para{Base 3}
In base 3, we now have 2 digits at our disposal, which allows our operations to have valid operands, however, there are not many combinations and many operations lead to the same answer. Table \ref{tab:b3} lists these values.

\begin{table}[H]
    \centering
    \tiny
    \renewcommand{\arraystretch}{1.4}
    \begin{tabular}{L L} 
        \textbf{\textrm{Increasing}} & \textbf{\textrm{Decreasing}}\\
        0_{10} = 0_{3}=\np~ & 0_{10}= 0_{3}= \np~\\
        1_{10} = 1_{3}=1^2 & 1_{10}= 1_{3}= 2-1\\
        2_{10}= 2_{3}= -1+2 \mbox{ or } 1\mul 2 & 2_{10}= 2_{3}= 2\mul 1 \mbox{ or } 2^1\\
        3_{10}= 10_{3}= 1+2 & 3_{10}= 10_{3}= 2+1\\
        5_{10}= 12_{3}= 12 & 7_{10}= 21_{3}=21\\
    \end{tabular}
    \caption{List of most of the possible base 3 numbers in increasing and decreasing sequential order.}
    \label{tab:b3}
\end{table}


\para{Base 4}
Base 4 is the smallest base where anything interesting happens and we can list a significant portion of integers with the largest number being $19683_{10}$ since in base 4 it is $3^{21}_{4}$. Table \ref{tab:b4} lists the first 20 values and then a few of interest.
    
\begin{table}[H]
    \centering
    \tiny
    \renewcommand{\arraystretch}{1.4}
    \begin{multicols}{2}
    \begin{tabular}{L L} 
        \textbf{\textrm{Increasing}} & \textbf{\textrm{Decreasing}}\\ 
        0_{10}=0_{4}= 1+2-3 & 0_{10}=0_{4}=3-2-1 \\
        1_{10}=1_{4}= 1^{2+3} & 1_{10}=1_{4}= 3-2\mul 1\\
        2_{10}=2_{4}= 1-2+3 & 2_{10}=2_{4}= 3-2+1\\
        3_{10}=3_{4}= 12-3 & 3_{10}=3_{4}= 3\mul (2-1)\\
        4_{10}=10_{4}= 1^2 +3& 4_{10}=10_{4}= 3+2-1\\
        5_{10}=11_{4}= -1+2\mul 3 & 5_{10}=11_{4}= 3+2\mul 1\\
        6_{10}=12_{4}= 1+2+3& 6_{10}=12_{4}= 3+2+1 \\
        7_{10}=13_{4}= 1+2\mul 3& 7_{10}=13_{4}= 3\mul 2+1 \\
        8_{10}=20_{4}= (1\mul 2)^3& 8_{10}=20_{4}= 3^2-1 \\
        9_{10}=21_{4}= 12+3& 9_{10}=21_{4}= 3\mul (2+1) \\
        
    \end{tabular}
    
    \begin{tabular}{L L} 
        \textbf{\textrm{Increasing}} & \textbf{\textrm{Decreasing}}\\
        10_{10}=22_{4}= -1+23& 10_{10}=22_{4}= 3^2+1\\
        11_{10}=23_{4}= 1\mul 23& 11_{10}=23_{4}=\np~ \\
        12_{10}=30_{4}= 1+23& 12_{10}=30_{4}=3+21 \\
        13_{10}=31_{4}= \np~& 13_{10}=31_{4}=32-1 \\
        14_{10}=32_{4}= \np~& 14_{10}=32_{4}=32\mul 1 \\
        15_{10}=33_{4}= \color{red}{(1+2)\ccat 3}& 15_{10}=33_{4}=32+1 \\
        16_{10}=100_{4}= \np~& 16_{10}=100_{4}=\np~ \\
        17_{10}=101_{4}= \np~& 17_{10}=101_{4}= \np~\\
        18_{10}=102_{4}= 12\mul 3& 18_{10}=102_{4}= \np~\\
        27_{10}=123_{4}= 123& 27_{10}=123_{4}= 3\mul 21\\
        57_{10}=321_{4}= \np~& 57_{10}=321_{4}= 321\\
    \end{tabular}
    \end{multicols}
    \caption{List of most of the possible base 4 numbers in increasing and decreasing sequential order.}    \label{tab:b4}
\end{table}


\section{Overview of 5-9}\label{sec:small}
For organizational reasons, we overview things of interest about bases 5 through 9, and the actual listings of the expressions are omitted for space with only the first 40 numbers shown for $5-7$ and the first 20 shown for 8 and 9. 

All possible results were generated (negatives, decimals, etc.), and everything could be listed rather than giving just the organized list as presented. However, the sheer number of results makes it infeasible to do so. For instance, with base eight, there are over 45,000 integer results, with most of them not being consecutive. 


\begin{table}[H]
    \centering
    \tiny
    \renewcommand{\arraystretch}{1.4}
    \setlength{\tabcolsep}{4pt}
    
\begin{multicols}{2}
\begin{tabular}{L  L  L L}
    \textbf{\normalsize $B_{10}$} & \textbf{\normalsize $B_5$}&\textbf{\normalsize Increasing}&\textbf{\normalsize Decreasing}\\
    
    0_{10}& 0_{5} & 1^2 + 3 -4& (4-3)-(2-1)\\
    1_{10}& 1_{5} & 1^{234}  &(4-3)\mul (2-1) \\
    2_{10}& 2_{5} & 1+2+3-4 & -4+3+2+1\\
    3_{10}& 3_{5} & 1+2\mul 3-4 &  4-3+2\mul 1 \\
    4_{10}& 4_{5} & 1+2-3+4 &  4+(3-2-1) \\
    5_{10}& 10_{5} & 1+2^3-4 & 4+3-2\mul 1 \\
    6_{10}& 11_{5} & 1-2+3+4 & 4+(3-2+1)\\
    7_{10}& 12_{5} & 1^2\mul 3+4 &  4+3\mul (2-1) \\
    8_{10}& 13_{5} & -1+2+3+4 &  4+3+2-1 \\
    9_{10}& 14_{5} & -1-2+(3\mul 4) &  4+3+2\mul 1\\ 
    10_{10}& 20_{5} & 1+2+3+4 &  4+3+2+1\\
    11_{10}& 21_{5} & 1+2\mul 3+4 & 4+3\mul 2+1\\
    12_{10}& 22_{5} & -12+34 &  4\mul 3\mul (2-1)\\
    13_{10}& 23_{5} & -1+2+3\mul 4 &  4+3\mul (2+1)\\
    14_{10}& 24_{5} & 1\mul 2+3\mul 4 &  4\mul 3+2\mul 1\\
    15_{10}& 30_{5} & 1+2+3\mul 4 &  4\mul 3+2\mul 1\\ 
    16_{10}& 31_{5} & -1+23+4  &  4\mul (3+2-1)\\
    17_{10}& 32_{5} & 1\mul 23+4  &  \np~\\
    18_{10}& 33_{5} & 1+23+4  &  4+3+21\\ 
    19_{10}& 34_{5} & -1+(2+3)\mul 4 &  4\mul (3+2)-1\\
    
\end{tabular}
    
\begin{tabular}{L  L  L L}
    \textbf{\normalsize $B_{10}$} & \textbf{\normalsize $B_5$}&\textbf{\normalsize Increasing}&\textbf{\normalsize Decreasing}\\
    20_{10}& 40_{5} & 1\mul (2+3)\mul 4 &  4\mul (3+2)\mul 1 \\
    21_{10}& 41_{5} & (1+2)\mul (3+4) & 4+32\mul 1\\
    22_{10}& 42_{5} & 1+2+34  & 4+32+1\\
    23_{10}& 43_{5} & -1+2\mul 3\mul 4 &  4\mul 3+21\\
    24_{10}& 44_{5} & 1\mul 2\mul 3\mul 4 & 4\mul 3\mul 2\mul 1\\
    
    25_{10}& 100_{5} & 1+2\mul 3\mul 4 &  4\mul 3\mul 2+1\\ 
    26_{10}& 101_{5} & 12+34 & 43+2+1 \\ 
    27_{10}& 102_{5} & \np~ &  \np~\\
    28_{10}& 103_{5} & (1+2\mul 3)\mul 4 &  4\mul (3\mul 2+1)\\
    29_{10}& 104_{5} & \color{red}{(1\mul 2 + 3)\ccat 4} &  -4+3\mul 21\\
    30_{10}& 110_{5} & \color{red}{1\ccat (-2+ 3 + 4)} &  \color{red}{(5-4)\ccat (3\mul 2 - 1)}\\
    31_{10}& 111_{5} & -1+2^3 \mul 4  & 4+3^{2+1}\\
    32_{10}& 112_{5} & 1\mul 2^3\mul 4 &  4\mul (3^2-1)\\
    33_{10}& 113_{5} & 1+2^3\mul 4 &  (4^3)/2+1 \\
    34_{10}& 114_{5} & 123-4 &  43+21\\
    35_{10}& 120_{5} & \color{red}{1\ccat(-2+ 3 \mul 4)} &  4\mul 3^2-1\\
    36_{10}& 121_{5} & (1+2)\mul 3\mul 4 &  4\mul 3\mul (2+1)\\
    37_{10}& 122_{5} & -1+2\mul 34  &  4+3\mul 21\\
    38_{10}& 123_{5} & 1\mul 2\mul 34 &  \np~\\
    39_{10}& 124_{5} & 1+2\mul 34  &  \np~\\
\end{tabular}
\end{multicols}
\vspace*{.2cm}
\caption{List of base 5 numbers from 0 to 39 in increasing and decreasing sequential order.}
\label{tab:b5}
\end{table}


\para{Base 5} 
There are four numbers in the representation for base five, and thus there is enough variability to begin making a meaningful amount of different combinations and possible integers. Still, this may be considered a relatively small base since the first impossible integer is 27. We also filled in some of the gaps with explicit concatenation.  
Table \ref{tab:b5} shows a list of the positive integers $0-39$ with their representations. The missing ones are not possible.

\para{Base 6}
Each increase in base exponentially increases the number of possibilities and the first positive integers that can not be expressed are 67 (increasing) and 77 (decreasing), and 97 is the first one not representable by either.
Table \ref{tab:b6} shows a list of the positive integers $0-39$ with their representations.

\begin{table}[h]
    \centering
    \tiny
    \renewcommand{\arraystretch}{1.4}
    \setlength{\tabcolsep}{3.5pt}
    \setlength\columnsep{30pt}
\begin{multicols}{2}
\begin{tabular}{L  L  L L}
    \textbf{\normalsize $B_{10}$} & \textbf{\normalsize $B_6$}&\textbf{\normalsize Increasing}&\textbf{\normalsize Decreasing}\\
       0_{10}& 0_{6} & 1^{23}+4-5 & 5-4-3+2\mul 1\\
    1_{10}& 1_{6} & 1^{2345}  & 5-4^{3-2-1}\\
    2_{10}& 2_{6} & 12+3-4-5 & 5+4-3\mul 2-1\\
    3_{10}& 3_{6} & 1-2+3-4+5 & 5+4-3-2-1  \\
    4_{10}& 4_{6} & 12-3+4-5  & 5+4-3-2\mul 1\\
    5_{10}& 5_{6} & 1+2+3+4-5 & 5+4-3-2+1\\
    6_{10}& 10_{6} & 12-3-4+5 & 5+4-3\mul (2-1)  \\
    7_{10}& 11_{6} & 1+2+3-4+5  & 5+4-3+2-1\\
    8_{10}& 12_{6} & 1+2\mul 3-4+5 & 5+4-3+2\mul 1\\
    9_{10}& 13_{6} & 1+2-3+4+5 &  5+4+3-2-1 \\
    10_{10}& 14_{6} & 12+3+4-5  & 5+4+3-2\mul 1\\
    11_{10}& 15_{6} & 12-3\mul (4-5)  & 5+4+3-2+1\\
    12_{10}& 20_{6} & 12+3-4+5 &   5+4+3\mul (2-1)\\
    13_{10}& 21_{6} & 1-2\mul (3-4-5) & 5+4+3+2-1  \\
    14_{10}& 22_{6} & 12-3+4+5  & 54-32\mul 1\\
    15_{10}& 23_{6} & 1+2+3+4+5 & 5+4+3+2+1\\
    16_{10}& 24_{6} & 1+2\mul 3+4+5 &  5+4+3\mul 2+1 \\
    17_{10}& 25_{6} & 1+23-4+5  & 5+4+3^2-1\\
    18_{10}& 30_{6} & 1+2^3+4+5 & 5+4+3\mul (2+1)\\
    19_{10}& 31_{6} & 1\mul 2+34-5 &  5+4+3^2+1 \\
    
\end{tabular}
    
\begin{tabular}{L  L  L L}
    \textbf{\normalsize $B_{10}$} & \textbf{\normalsize $B_6$}&\textbf{\normalsize Increasing}&\textbf{\normalsize Decreasing}\\
    20_{10}& 32_{6} & 12+3+4+5 & 5+4\mul 3+2+1\\
    21_{10}& 33_{6} &  1^{23}+4\mul 5 & 5+4\mul (3+2-1)\\
    22_{10}& 34_{6} & 123-45 & 5-4+32+1\\
    23_{10}& 35_{6} & 12\mul 3+4-5 &  5\mul 4+3\mul (2-1) \\
    24_{10}& 40_{6} &  (12/3)\mul (4+5) & 54+3-21\\
    
    25_{10}& 41_{6} & 12+34-5 & 54-3\mul (2+1)\\
    26_{10}& 42_{6} & 1+2+3+4\mul 5 & 5\mul 4+3+2+1  \\
    27_{10}& 43_{6} & 1+2\mul 3+4\mul 5  & 5\mul 4+3\mul 2+1\\
    28_{10}& 44_{6} & 1\mul 2-3+45 & 54-3-2-1\\
    29_{10}& 45_{6} & 1+2-3+45 &  54-3-2\mul 1 \\
    30_{10}& 50_{6} & 12\mul (3/4)\mul 5 &54-3-2+1  \\
    31_{10}& 51_{6} & 123-4\mul 5 & 54-3\mul (2-1)\\
    32_{10}& 52_{6} & 12\mul (3-4+5) & 54-3+2-1\\
    33_{10}& 53_{6} & 12\mul 3+4+5  &  54-3+2\mul 1 \\
    34_{10}& 54_{6} & 12-3+45  & 54+3-2-1\\
    35_{10}& 55_{6} & 12+34+5 & 54+3-2\mul 1\\
    36_{10}& 100_{6} & 1+23+4\mul 5 & 54+3-2+1  \\
    37_{10}& 101_{6} & 1\mul 2+(3+4)\mul 5  & 54+3\mul (2-1)\\
    38_{10}& 102_{6} & 1+2+(3+4)\mul 5 & 54+3+2-1\\
    39_{10}& 103_{6} & 1\mul 2\mul 34-5 &  54+3+2\mul 1 \\
\end{tabular}

\end{multicols}
\vspace*{.2cm}
\caption{List of base 6 numbers from 0 to 39 in increasing and decreasing sequential order.}
\label{tab:b6}
\end{table}


\para{Base 7} 
Starting with base 7, the amount of numbers possible explodes, and thus, we will simply list the numbers without trying to fit them onto a single page. In fact, every number is expressible until 260. Curiously the first inexpressible decreasing integer is 262.
Table \ref{tab:b7} shows a list of the positive integers $0-39$ with their representations.

\begin{table}[H]
    \centering
    \tiny
    \renewcommand{\arraystretch}{1.4}
    \setlength{\tabcolsep}{3.1pt}
    \setlength\columnsep{30pt}
\begin{multicols}{2}
\begin{tabular}{L  L  L L} 
    \textbf{\normalsize $B_{10}$} & \textbf{\normalsize $B_7$}&\textbf{\normalsize Increasing}&\textbf{\normalsize Decreasing}\\
       0_{10}& 0_{7} & 1^{234}+5-6 & 6+5-4-3\mul 2 -1\\
    1_{10}    &  1_{7}     &  12-3-4+5-6           &  654^{3-2-1}   \\
    2_{10}    &  2_{7}     &  12-3+4\mul(5-6)         &  6+5-4-3-2\mul 1         \\
    3_{10}    &  3_{7}     &  12-3-4-5+6           &  6+5+4+3-21          \\
    4_{10}    &  4_{7}     &  12+34-5\mul 6            &  6+54/3-21           \\
    5_{10}    &  5_{7}     &  12+3+4-5-6           &  6+5+4-3^2-1        \\
    6_{10}    &  6_{7}     &  12+3+(4-5)\mul 6         &  65+4-3\mul 21           \\
    7_{10}    &  10_{7}    &  12+3-4+5-6           &  65-4\mul(3^2+1)       \\
    8_{10}    &  11_{7}    &  123/45+6             &  6+5+4-3\mul 2-1         \\
    9_{10}    &  12_{7}    &  12+3-4-5+6           &  6+5+4-3-2-1         \\
    10_{10}   &  13_{7}    &  12+3\mul 4-5-6           &  65-4\mul 3^2-1         \\
    11_{10}   &  14_{7}    &  12+3/(4+5)\mul 6         &  65-4\mul 3\mul(2+1)        \\
    12_{10}   &  15_{7}    &  123/45\mul 6             &  65-4\mul 3^2+1         \\
    13_{10}   &  16_{7}    &  12+3+(4-5)^6        &  65-43-2-1           \\
    14_{10}   &  20_{7}    &  12-34+5\mul 6            &  65-43-2\mul 1           \\
    15_{10}   &  21_{7}    &  12+3+4+5-6           &  65-43-2+1           \\
    16_{10}   &  22_{7}    &  12+3-4\mul(5-6)         &  65-43\mul(2-1)         \\
    17_{10}   &  23_{7}    &  12+3+4-5+6           &  65-43+2-1           \\
    18_{10}   &  24_{7}    &  12+3\mul(4+5-6)         &  65-43+2\mul 1           \\
    19_{10}   &  25_{7}    &  12+3-4+5+6           &  65-43+2+1           \\
    
\end{tabular}
\hspace*{-.5cm} 
\begin{tabular}{L  L  L L}
    \textbf{\normalsize $B_{10}$} & \textbf{\normalsize $B_7$}&\textbf{\normalsize Increasing}&\textbf{\normalsize Decreasing}\\
    20_{10}   &  26_{7}    &  12+34/5+6            &  65-4-32\mul 1           \\
    21_{10}   &  30_{7}    &  123-4-56             &  65-4-32+1           \\
    22_{10}   &  31_{7}    &  123-4\mul(5+6)          &  65-4\mul 3\mul 2-1          \\
    23_{10}   &  32_{7}    &  12+34-5-6            &  65-4\mul 3\mul 2\mul 1          \\
    24_{10}   &  33_{7}    &  123\mul 4/(5+6)          &  65+4-3^{2+1}       \\
    
    25_{10}   &  34_{7}    &  12-34+56             &  65-4-3-21           \\
    26_{10}   &  35_{7}    &  12+3+4\mul 5-6           &  65\mul 4-321            \\
    27_{10}   &  36_{7}    &  123-45-6             &  65+4-32-1           \\
    28_{10}   &  40_{7}    &  12\mul 3+(4-5)^6        &  65+4-32\mul 1           \\
    29_{10}   &  41_{7}    &  123+4-56             &  65+4-32+1           \\
    30_{10}   &  42_{7}    &  12-3\mul (4-5-6)         &  6+5-4+32\mul 1          \\
    31_{10}   &  43_{7}    &  12\mul 3-4\mul (5-6)         &  65-43+21            \\
    32_{10}   &  44_{7}    &  123-4-5\mul 6            &  65-4\mul 3-2-1          \\
    33_{10}   &  45_{7}    &  12+34+5-6            &  65+4-3-21           \\
    34_{10}   &  46_{7}    &  12-34\mul (5-6)          &  65-4-3\mul(2+1)        \\
    35_{10}   &  50_{7}    &  12+34-5+6            &  65-4-3^2+1         \\
    36_{10}   &  51_{7}    &  12+3^{4+5-6}      &  654/3^2-1          \\
    37_{10}   &  52_{7}    &  12/3+4+5\mul 6           &  654/3/(2+1)         \\
    38_{10}   &  53_{7}    &  12+3+4\mul 5+6           &  654/3^2+1          \\
    39_{10}   &  54_{7}    &  123-45+6             &  65+4+3-21           \\
\end{tabular}

\end{multicols}
\vspace*{.2cm}
\caption{List of base 7 numbers from 0 to 39 in increasing and decreasing sequential order.}
\label{tab:b7}
\end{table}

\vspace*{-.5cm}

\para{Base 8} 
Table \ref{tab:b8} shows a list of the positive integers $0-19$ with their representations. Due to the length of the expressions, there is not room for more numbers. Base eight does not have an inexpressible number until 614 for an increasing sequence, and 809 for a decreasing sequence. The first positive integer that can not be expressed by either is 1192.

\begin{table}[H]
    \centering
    \tiny
    \renewcommand{\arraystretch}{1.4}
    \setlength{\tabcolsep}{4pt}
    \setlength\columnsep{30pt}
\begin{multicols}{2}
\hspace*{-.5cm}
\begin{tabular}{L  L  L L}
    \textbf{\normalsize $B_{10}$} & \textbf{\normalsize $B_8$}&\textbf{\normalsize Increasing}&\textbf{\normalsize Decreasing}\\
          0_{10}& 0_{8} & 1^{2345}+6-7 & 7$+$6$-$5$-$4$-$3$-$2$+$1\\
        1_{10}    &  1_{8}     &  1^{234567}          &  7654^{3-2-1}          \\
    2_{10}    &  2_{8}     &  12+3-4+(5-6)\mul 7          &  76+5-4^3-2+1          \\
    3_{10}    &  3_{8}     &  12+3-4-5+6-7            &  76+5-4^3\mul (2-1)        \\
    4_{10}    &  4_{8}     &  123-45-6\mul 7              &  76+5-4^3+2-1          \\
    5_{10}    &  5_{8}     &  12+3-4-5-6+7            &  76+5-4^3+2\mul 1          \\
    6_{10}    &  6_{8}     &  12+34\mul (5-6)/7           &  76+5-4^3+2+1          \\
    7_{10}    &  7_{8}     &  123/4^56+7             &  76-5\mul (4$+$3\mul 2$+$1)         \\
    8_{10}    &  10_{8}    &  123-4\mul 5-67              &  76-54-3^2-1           \\
    9_{10}    &  11_{8}    &  12+3+4+5-6-7            &  76-54-3\mul (2+1)          \\
    10_{10}   &  12_{8}    &  123\mul 4-56\mul 7              &  76-54-3^2+1           \\
    11_{10}   &  13_{8}    &  123-4\mul (5+6+7)           &  76+5+4\mul (3-21)          \\
    12_{10}   &  14_{8}    &  12+3+4+5\mul (6-7)          &  76+5-4-3\mul 21            \\
    13_{10}   &  15_{8}    &  12+3+4-5-6+7            &  76-54-3-2\mul 1            \\
    14_{10}   &  16_{8}    &  123-4-5\mul (6+7)           &  76-54-3-2+1            \\
    15_{10}   &  17_{8}    &  12+34-5\mul 6+7             &  76+5-43-21             \\
    16_{10}   &  20_{8}    &  12+3+4+(5-6)^7         &  76-54-3+2-1            \\
    17_{10}   &  21_{8}    &  12+34/(5+6-7)           &  76-54-3+2\mul 1            \\
    18_{10}   &  22_{8}    &  12+3-45+6\mul 7             &  76-54+3-2-1            \\
    19_{10}   &  23_{8}    &  123-4-5-67              &  76-54+3-2\mul 1            \\
   
 \end{tabular}
 
 \hspace*{-.5cm}   
 \begin{tabular}{L  L  L L}
     \textbf{\normalsize $B_{10}$} & \textbf{\normalsize $B_8$}&\textbf{\normalsize Increasing}&\textbf{\normalsize Decreasing}\\
        
     20_{10}   &  24_{8}    &  12+34-5-6-7             &  76+5+4-3\mul 21            \\
    21_{10}   &  25_{8}    &  123-4\mul 5-6\mul 7             &  76-54+3\mul (2-1)          \\
    22_{10}   &  26_{8}    &  123+4-5\mul (6+7)           &  76-54+3+2-1            \\
    23_{10}   &  27_{8}    &  12+3+4+5-6+7            &  76-54+3+2\mul 1            \\
    24_{10}   &  30_{8}    &  12+3-4\mul (5-6)+7          &  76-54+3+2+1            \\
        25_{10}   &  31_{8}    &  12+3+4-5+6+7            &  76-54+3\mul 2+1            \\
    26_{10}   &  32_{8}    &  123-4-56-7              &  76-54+3^2-1           \\
    27_{10}   &  33_{8}    &  123+4-5-67              &  76+5-4\mul (3^2+1)        \\
    28_{10}   &  34_{8}    &  123+(4-5)\mul 67            &  76-54+3^2+1           \\
    29_{10}   &  35_{8}    &  123-4+5-67              &  76+5-43-2-1            \\
    30_{10}   &  36_{8}    &  12+34+5-6-7             &  76+5-43-2\mul 1            \\
    31_{10}   &  37_{8}    &  12+34+(5-6)\mul 7           &  76+5-43-2+1            \\
    32_{10}   &  40_{8}    &  123-4-5-6\mul 7             &  76+5-43\mul (2-1)          \\
    33_{10}   &  41_{8}    &  123-45-6-7              &  76+5-43+2-1            \\
    34_{10}   &  42_{8}    &  123+4-56-7              &  76+5-43+2\mul 1            \\
    35_{10}   &  43_{8}    &  12+3+4+5+6+7            &  76+5-43+2+1            \\
    36_{10}   &  44_{8}    &  12+3\mul (4+5)+6-7          &  76+5-4-32-1            \\
    37_{10}   &  45_{8}    &  123+4+5-67              &  76+5-4-32\mul 1            \\
    38_{10}   &  46_{8}    &  12+3\mul (4+5)-(6-7)        &  76+5-4-32+1            \\
    39_{10}   &  47_{8}    &  123-4/5\mul 67              &  76+5-4\mul (3\mul 2+1)         \\
\end{tabular}
\end{multicols}
\caption{List of base 8 numbers from 0 to 47 in increasing and decreasing sequential order.}
\label{tab:b8}
\end{table}


\vspace*{.2cm}
\para{Base 9} 
Similar to base eight, only representations for numbers $0-19$ are shown in Table \ref{tab:b9}. The first unrepresentable positive integers for increasing and decreasing sequential representations are 3293 and 4570, respectively. The integer 5414 is the smallest positive integer unrepresentable by either.


\begin{table}[H]
    \centering
    \tiny
    \renewcommand{\arraystretch}{1.4}
    \setlength{\tabcolsep}{5pt}
    \setlength\columnsep{50pt}
\begin{multicols}{2}
\hspace*{-.7cm} 
\begin{tabular}{L  L  L L}
    \textbf{\normalsize $B_{10}$} & \textbf{\normalsize $B_9$}&\textbf{\normalsize Increasing}&\textbf{\normalsize Decreasing}\\
    0_{10}         &  0_{9}          &  1^{23456}$+$ 7$-$8        &  8$-$7$-$6$+$5$+$4$-$3$-$2$+$1          \\ 
    1_{10}         &  1_{9}          &  1^{2345678}      &  8$+$7$-$6$-$5$-$4$+$3$-$2\mul 1         \\ 
    2_{10}         &  2_{9}          &  123$-$4\mul (5$+$6)$+$7\mul 8    &  876$-$ 5\mul ((4\mul 3)^2$-$1)   \\ 
    3_{10}         &  3_{9}          &  (123$+$45)/(6$+$7)$-$8       &  876/(5^{43})$+$2$+$1      \\ 
    4_{10}         &  4_{9}          &  (123$-$4$+$5$+$67)/8       &  (87$+$6$-$54)/(3\mul (2$+$1))    \\ 
    5_{10}         &  5_{9}          &  (123$+$4$-$5$+$67)/8       &  (876$-$5^4$+$3)/21      \\ 
    6_{10}         &  6_{9}          &  (123$+$4$-$56$+$7)/8       &  (87$+$65)/(4\mul 3\mul 2$-$1)    \\ 
    7_{10}         &  7_{9}          &  123$+$4$+$5$-$((6$+$7)\mul 8)    &  876$-$(5$+$4)^3$+$21      \\ 
    8_{10}         &  8_{9}          &  1234/(5^{67})$+$8         &  (87$+$6$+$5)/(4$+$3\mul 2)$-$1  \\ 
    9_{10}         &  10_{9}         &  (123$-$45)/67$+$8          &  87$+$65$-$4^3\mul 2 $+$1   \\ 
    10_{10}        &  11_{9}         &  (123$+$(4\mul 5$-$6\mul 7))/8    &  87$+$65$-$4^3\mul 2\mul 1   \\ 
    11_{10}        &  12_{9}         &  (123$+$4$-$5$+$6$+$7)/8 & 87$+$65$-$4^3\mul 2$+$1   \\ 
    12_{10}        &  13_{9}         &  123$+$(4$-$5)\mul 6\mul (7$+$8)    &  87$+$6$+$5$-$43\mul 2\mul 1    \\ 
    13_{10}        &  14_{9}         &  (1234$-$5$+$6)/78          &  87$+$6$+$5$-$43\mul 2$+$1    \\ 
    14_{10}        &  15_{9}         &  (123$+$4$+$5$-$6$+$7)/8    &  (87$+$6$+$5$-$4^3)/2$+$1  \\ 
    15_{10}        &  16_{9}         &  123\mul (4/5$+$6)^{7-8}   &  (87$+$6$+$5\mul 4)/(3\mul 2$+$1)  \\ 
    16_{10}        &  17_{9}         &  ((123$+$4\mul 5)/67)\mul 8       &  87$+$6$+$(5$-$43)\mul 2$-$1    \\ 
    17_{10}        &  18_{9}         &  (1234$+$5\mul 6)/(7\mul 8)       &  87$+$6$+$(5$-$43)\mul (2\mul 1)    \\ 
    18_{10}        &  20_{9}         &  (123$+$45$-$6$+$7)/8       &  87$+$65$-$(4\mul (32$+$1))       \\ 
    19_{10}        &  21_{9}         &  ((123$+$45)/(6$+$7))$+$8       &  (876$+$5)/((43$-$2)$+$1)       \\ 
 \end{tabular}

 \hspace*{-.7cm} 
\begin{tabular}{L  L  L L}
     \textbf{\normalsize $B_{10}$} & \textbf{\normalsize $B_9$}&\textbf{\normalsize Increasing}&\textbf{\normalsize Decreasing}\\
     20_{10}        &  22_{9}         &  123$+$4$-$(5\mul 6)$+$(7\mul 8)    &  87$+$65\mul (4$-$3$+$2\mul 1)    \\ 
    21_{10}        &  23_{9}         &  123$+$4$+$5$-$(6\mul (7$+$8))   &  87$+$(65$-$(43\mul (2$+$1)))       \\ 
    22_{10}        &  24_{9}         &  123$+$(4\mul (56$-$78))          &  87$+$(65$-$(4\mul (32\mul 1)))       \\ 
    23_{10}        &  25_{9}         &  (123$+$4$+$5$-$6)/7$+$8    &  (876$-$(5^4))/(3$+$(2$-$1))   \\ 
    24_{10}        &  26_{9}         &  123$+$(4$-$(5$+$(6$+$78)))       &  ((87$+$65)/(4\mul 3))\mul 2$+$1    \\ 
    25_{10}        &  27_{9}         &  123$+$((4$-$5)\mul (6$+$78))       &  87$+$(6$+$((5$-$((4^3)$+$2))$+$1))  \\ 
    26_{10}        &  28_{9}         &  (123$-$4)$+$(5$-$(6$+$78))       &  87$+$(65$-$(4\mul (32$-$1)))       \\ 
    27_{10}        &  30_{9}         &  123$+$((4\mul (5$-$6))$-$78)       &  87$+$(6$+$((5$-$(4^3))$+$(2$-$1)))  \\ 
    28_{10}        &  31_{9}         &  (123$-$(4$+$5))$+$(6$-$78)       &  (87$+$((6$+$(5$+$4))/3))/(2$+$1)  \\ 
    29_{10}        &  32_{9}         &  (123$-$(4$+$((5$+$6)\mul 7)))$+$8    &  87$+$(6$+$(5$-$(4$+$(3\mul 21))))    \\ 
    30_{10}        &  33_{9}         &  123$+$((4$+$5)\mul ((6$-$7)\mul 8))    &  (876$-$(54\mul 3))/21          \\ 
    31_{10}        &  34_{9}         &  123$-$(4$+$(5$+$(6$+$(7\mul 8))))    &  ((87$+$(6$+$5))/((4$-$3)$+$2))$+$1  \\ 
    32_{10}        &  35_{9}         &  123$+$(4$-$(5$+$(67$+$8)))       &  ((87$+$65)/4)$-$((3/2)$+$1)    \\ 
    33_{10}        &  36_{9}         &  123$+$((4$-$5)\mul (67$+$8))       &  (876$+$(5$+$4))/(3$+$21)       \\ 
    34_{10}        &  37_{9}         &  123$+$(4$+$(5$-$(6$+$78)))       &  ((87$+$(65\mul 4))/(3^2))$-$1   \\ 
    35_{10}        &  38_{9}         &  123$+$(4$+$((5$-$6)\mul 78))       &  (876$-$(54$+$3))/21          \\ 
    36_{10}        &  40_{9}         &  123$+$((4$-$5)$+$(6$-$78))       &  ((87$+$(65$-$(4^3)))/2)$-$1   \\ 
    37_{10}        &  41_{9}         &  123$+$((4$-$((5$+$6)\mul 7))$+$8)    &  (87$+$(65$-$(4^3)))/(2\mul 1)   \\ 
    38_{10}        &  42_{9}         &  123$+$(4\mul (5$-$(6$+$(7$+$8))))    &  (876$+$5)\mul ((4$-$3)/21)       \\ 
    39_{10}        &  43_{9}         &  123$+$(45$-$((6$+$7)\mul 8))       &  (87$+$(6$+$(5$-$(4\mul 3))))/(2\mul 1)  \\ 
\end{tabular}
\vspace*{.2cm}
\end{multicols}
\caption{List of base 9 numbers from 0 to 19 in increasing and decreasing sequential order.}
\label{tab:b9}
\vspace*{-.5cm}
\end{table}

\section{Base 10} \label{sec:b10}
Taneja showed that crazy increasing sequential representations for base 10 numbers was possible for all numbers to 11111 with one exception \cite{Taneja:2013:1302.1479v5}. There is no known solution to 10958 with the numbers in increasing order. It is possible to get close, but not exact. We found two extremely close solutions. 
\begin{align}
10957.9775 &= -1+2^{(3^4)/5}/(6+7/8)+9 \\
10958.0021   &= (1+((2-(-3^{-4})^{5/(6\mul -7)})^{-8}))+9 
\end{align}
No closer solutions are possible. Running an exhaustive algorithm to look at all possible combinations yields no solution of 10958. Table \ref{tbl:10958} lists all values, and an expression yielding that value (there are many), that were found within the range $[10957.9, 10958.1]$. 

\begin{table}[t]
\centering \footnotesize
\renewcommand{\arraystretch}{1.3}
    \begin{tabular}{| L | L |}\hline 
        \textbf{Number} & \textbf{Expression} \\ \hline
        
        10957.90411   & -1+((2- 3\mul 4)^5)/(((6-7)/8)-9) \\ 
        10957.92857   &  -1/2\mul 3\mul (4^5\mul (6/7-8)+9)       \\ 
        10957.93277   & (-1/(2+3)+(4^{5-((6-7)/8)}))\mul 9 \\ 
        10957.97006   & -(1/2)+(34+((5^{-(6/7)+8})/9)) \\ 
        10957.97751   &  -1+2^{3^4/5}/(6+7/8)+9                   \\ 
        10958.00206   & ((1+((2-((-3^{-4})^{5/(6\mul -7)}))^{-8}))+9) \\ 
        10958.06611   & -1-2/3+ 4^{5-(6-7)/8}\mul 9 \\ 
        10958.09749   & (1+234)\mul (5\mul (6+((7/8)^{-9}))) \\ 
        \hline
    \end{tabular}              
\caption{List of base 10 numbers and the expressions that are close to 10958. This shows numbers within $.1$ of the desired number, and are rounded to five decimal digits due to the precision limitations in the calculations.}
\label{tbl:10958}
\end{table}

However, the original author uses concatenation without defining it as one of the allowable operators between operands as a step. Matt Parker found a solution if concatenation is allowed to occur as a step of the calculation, which is not done for any other number in the original paper. Let $\ccat$ be the concatenation operator. Thus, $234$ would be shown as $2\ccat 3\ccat 4$ in an equation. His solution is shown in \cite[Matt Parker]{numberphile1,numberphile2} and is:
\begin{equation}
10958 = 1 \mul 2\ccat 3 + ((4 \mul 5 \mul 6)\ccat 7  + 8) \mul 9.
\end{equation}
                                     
There are many other solutions if adding a new operator is allowed, such as factorials. Some examples are
\begin{align}
    10958 &= (1-2+3)\mul (456+7!-8-9),\text{\cite[Emmanuel Vantieghem]{primepuzzle}},\\
    10958 &=1+2+3!!+(-4+5!+6-7)\mul 89, \text{\cite[Inder J. Taneja]{primepuzzle}},     \\
    10958 &=1\mul 2\mul (3!!-4!\mul (5+6)+7!-8-9), \text{\cite[Inder J. Taneja]{primepuzzle}}, \mbox{ and}\\
    10958 &= -(1 + 2 - 3 + 4 - ((5! + 6) \mul  (78 + 9))), \text{\cite[Chris Smith]{primepuzzle}},
\end{align}

and it is possible if using the number $10$ as shown by Taneja, 
$10958=1*2^3+(4^5+6+7*8+9)*10$ \cite[Inder J. Taneja]{primepuzzle}.
Our approach settles this definitively through brute force search without the use of concatenation as a later step or another operation allowed beyond the initial ones. Thus, 10958 is the smallest integer for which this is not possible.

\begin{lemma} 
The integer 10958 is not expressible in base 10 increasing sequential representation (numbers $1-9$) with only the operations addition, multiplication, division, exponentiation, an initial concatenation of the numbers, arbitrary parentheses for operator precedence, and negation.
\end{lemma}

\begin{proof}
The proof is the program and its output of all combinations possible and their evaluation. The source code is available, and can be viewed (albeit in shortened form) in Appendix \ref{app:source}. 
\end{proof}


\section{Some More Fun}\label{sec:fun}
Here we look at several interesting open problems or additional ways to explore this concept.

\para{Fun Number Forms} Taneja gives a few in his paper for base 10, and we extend this with a few examples of numbers that are always expressible in a given base $b$.
\begin{itemize} \small \setlength\itemsep{.1em}
\item $0 = 1^{1\mul 2 \mul \dots \mul (b-3)}+(b-2)-(b-1)$ 
\item $0 = 1^{ 2  \dots (b-3)}+(b-2)-(b-1)$ 
\item $1 = 1^{1\mul 2 \mul \dots \mul (b-1)}$ 
\item $1 = 1^{12 \dots (b-1)}$ 
\item $b-1 = 1^{1\mul 2 \mul \dots \mul (b-2)} \mul (b-1)$
\item $b-1 = 1^{12\dots(b-2)}\mul (b-1)$ 
\item $b = 1^{1\mul 2 \mul \dots \mul (b-2)}+(b-1)$
\item $b = 1^{12\dots(b-2)}+(b-1)$                     
\end{itemize}
If $b$ is odd, then $0 = (b-1)-(b-2)-(b-3)+(b-4)-\dots+4-3-2+1$, and similar if $b$ is even.

\para{Taneja Primes}
Based on his work related to these numbers, we define a Taneja prime to be any prime expressible in crazy sequential representation for a base $b$.
Here, we investigate some interesting questions.
\begin{itemize} \small \setlength\itemsep{.1em}
    \item What is the smallest prime not expressible in a given base?
    \item What is the largest prime expressible in a given base?
    \item What is the sequence of primes not expressible (or expressible) in a given base?
    \item What is the characteristic function for the expressible or non-expressible primes for a base? 
    \item What is the sequence of integers (or primes) not expressible by either an increasing or decreasing representation.
    \item What is the smallest base a given prime (or integer) can be expressed in for increasing and decreasing? 
\end{itemize}

Table \ref{tab:overview} lists the first prime not expressible in a given base for increasing and decreasing representations as well as the first prime not expressible by either. For an increasing representation, Table \ref{tab:sminc} is the smallest base a prime can be sequentially represented in as well as an expression giving the value. 

\begin{table}[t]
    \centering
    \begin{tabular}{| c | c | c | c |}
        \hline
        \textbf{Base} & \textbf{Inc Prime} & \textbf{Dec prime} & \textbf{Both prime} \\ \hline
        3 & 7 & 5  & 11\\ \hline
        4 & 13& 11  & 17\\ \hline
        5 & 27& 17 & 27\\ \hline
        6 & 67& 83  & 97\\ \hline
        7 & 281 & 379 & 499 \\ \hline
        8 & 1153 & 809  & 1579 \\ \hline
        9 & 4597 & 5417  & 7027 \\ \hline
        10 & 15971 &  18493 & 25763\\ \hline        
    \end{tabular}
    \caption{A brief overview of the first primes not expressible in a given base for increasing and decreasing representations as well as the first prime not expressible by either.}
\label{tab:overview}
\end{table}

\begin{table}[t]
\centering
\scriptsize
\begin{multicols}{2}

\begin{tabular}{L c L} 
\renewcommand{\arraystretch}{1.1}
    \textbf{$B_{10}$} & \textbf{Base} & \textbf{Increasing} \\          
   2_{10}  & 3 & 1 \mul 2                                        \\
   3_{10}  & 3 & 1 + 2                                           \\
   5_{10}  & 3 & 12                                              \\
   7_{10}  & 4 & 1 + 2 \mul 3                                    \\
   11_{10} & 4 & 1 + 23                                          \\
   13_{10} & 5 & -1 + 2 + 3 \mul 4                               \\
   17_{10} & 5 & 1 + 2                                           \\
   19_{10} & 5 & -1 + (2 + 3) \mul 4 	                          \\
   23_{10} & 5 & -1 + 2 \times 3 \mul 4	                      \\
   29_{10} & 5 & \color{red}{(1\mul 2 + 3)\ccat 4}               \\
   31_{10} & 5 & -1+2^3 \mul 4	                                  \\
   37_{10} & 5 & -1+2\mul 34	                                  \\
   41_{10} & 6 & (1+2)\mul 3\mul 4+5	                          \\
   43_{10} & 6 & 12+(3+4)\mul 5	                              \\
   47_{10} & 6 & 1+2\mul (3+4\mul 5)	                          \\
   53_{10} & 5 & 1+23\mul 4	                                  \\
   59_{10} & 6 & 1-2+3\mul 4\mul 5	                              \\
   61_{10} & 6 & 12\mul (3+4)+5	                              \\
   67_{10} & 6 & \color{red}{1 \ccat(2+3)\ccat (-4 + 5 )}	      \\
   71_{10}  & 6 & 1+2\mul (3+4)\mul 5v                \\
   73_{10}  & 6 & 1+2^{3}\mul (4+5)	                 \\
\end{tabular}
   
\begin{tabular}{L c L} 
    \textbf{$B_{10}$} & \textbf{Base} & \textbf{Increasing} \\   
   79_{10}  & 5 & -1\mul 2+3^4 	                     \\
   83_{10}  & 5 & 1\mul 2+3^4	                     \\
   89_{10}  & 6 & 1+2+3^4+5	                         \\
   97_{10}  & 7 & 12\mul 3\mul 4-5-6  	             \\
   101_{10} & 6 & 12\mul 3\mul 4+5   	             \\
   103_{10} & 6 & (1+2)^3\mul 4-5 	                 \\
   107_{10} & 6 & -1-2+34\mul 5  	                 \\
   109_{10} & 6 & 1-2+34\mul 5 	                     \\
   113_{10} & 6 & 1+2+34\mul 5 	                     \\
   127_{10} & 5 & -1+2^{3+4}  	                     \\
   131_{10} & 7 & 12+3^4+56  	                     \\
   137_{10} & 6 & 1^2\mul 345  	                     \\
   139_{10} & 6 &  1\mul 2+345  	                 \\
   149_{10} & 7 & 12\mul 3\mul 4+56 	             \\
   151_{10} & 6 & -1+2\mul (3^4-5)  	             \\
   157_{10} & 6 & 1\mul 2\mul 3^4-5  	             \\
   163_{10} & 5 & 1+2\mul 3^4 	                     \\
   167_{10} & 6 & 1\mul 2\mul 3^4+5 	             \\
   173_{10} & 6 & 1+2\mul (3^4+5)                    \\
\end{tabular}

\end{multicols}
\caption{Smallest base a prime is sequentially representable in for an increasing representation.} \label{tab:sminc}
\end{table}



\para{Limited Operations}
The flexibility gained in sequential representations as the base gets larger is evident, and will continue for larger bases. Each new number exponentially increases the number of combinations.

Of interest would be to prove some estimates about the first number not expressible for a given base under certain operations such as just addition/subtraction, just addition/multiplication, just concatenation/exponentiation, etc.

\begin{itemize} \small \setlength\itemsep{.1em}
 \item How many unique numbers, given the operations above, can a given base generate?
 \item How many integers, given the operations above, can a given base generate?
 \item How many ways, given the operations above, can a given number be uniquely represented sequentially (ignoring parenthetical differences)?
 \item For $n \ge 6$, is the first missing decreasing number always greater than the first missing increasing number? Is there a way to determine if increasing or decreasing will not express a number first?
 \item Does a sequential representation exist with a set of operations $O$ in base $b$?
 \item For a given number, what bases can represent it?
\end{itemize}

\para{Continuing Problems}
All of the listed problems so far are also open to questions about the complexity. What is the complexity of finding the smallest base that a number $N$ can be sequentially represented in? To slightly extend the question, given an $N$, what bases can it be represented in? Further, given the computational domain, what is the solution to some of these questions if limited strictly to integer arithmetic?

There are many more open questions related to this problem in recreational mathematics. Such as noting that we, along with the original authors, focused solely on positive integers. All of these questions are open for rational, real, or other sets of numbers. 

\section*{Acknowledgements}
The author would like to thank those that have already given feedback on the paper, which helped improve the quality substantially. Most notably the dynamic programming improvements. The original version was optimized for limited memory systems, but took significantly longer.












\newpage
\appendix

\section{Minimal Source Code} \label{app:source}
The source code listed here has been shortened for clarity (with acknowledgements to others). For base 10, it takes hours to compute the dictionary and over 50 Gigabytes of memory to run, so saving to a file may be more efficient in order to not recompute the dictionary everytime. However, brute forcing using limited memory takes significantly longer, but can be done. A text file with a single representation for every value is about 20 Gigabytes for base 10.

\begin{lstlisting}[language=Python]

def FindExpressions(i, j, base=10):
    #create dictionary
    D={}
    #concatenate numbers i to j in the given base
    nt = int("".join([str(x) for x in range(i, j+1)]),base)
    #key is number between i,j in base and the value is the string
    D[float(nt)] = str(nt)
    #store negative version
    D[-float(nt)] = "-" + str(nt)
    #base case is i==j or i>j
    if (i != j):
        for k in range(i, j):
            #get optimal dictionary left of k from i thru k
            DL = FindExpressions(i, k, base)
            #get optimal dictionary right of k from k+1 to j
            DR = FindExpressions(k+1, j, base)           
            #all ways to combine optimal left/right 
            for x in DL:
                for y in DR:
                    D[x+y] = "(" + DL[x] + "+" + DR[y] + ")"
                    D[x-y] = "(" + DL[x] + "-" + DR[y] + ")" 
                    D[x*y] = "(" + DL[x] + "*" + DR[y] + ")"
                    if  y != 0:
                        D[x/y] = "(" + DL[x] + "/" + DR[y] + ")"
                    try:
                        D[pow(x,y)] = "(" + DL[x] + "^" + DR[y] + ")"
                    except:
                        pass
    return D

if __name__ == "__main__":
    #for each base 3-10   
    for base in range(2,11):
        D = FindExpressions(1, base-1, base)
        #find the first k numbers not in the dictionary
        k = 10
        i = 0.
        while k > 0:
            while i in D:
                i = i + 1.
            #print the number it couldn't find
            print(i)
            i = i + 1.
            k = k - 1
        
\end{lstlisting}

\end{document}